\documentclass[a4paper,11pt,draft]{amsart}
\usepackage{eucal}
\usepackage{amsmath,amsthm,amssymb}
\usepackage{amsfonts}
\usepackage{latexsym}
\usepackage[dvips]{graphics}
\usepackage{amsrefs}
\usepackage{mathrsfs}
\theoremstyle{plain}
\newtheorem{thm}{Theorem}[section]
\newtheorem{prop}[thm]{Proposition}
\newtheorem{lemma}[thm]{Lemma}
\newtheorem{cor}[thm]{Corollary}

\theoremstyle{definition}
\newtheorem{defn}[thm]{Definition}
\newtheorem{exa}[thm]{Example}

\theoremstyle{remark}
\newtheorem{rem}[thm]{Remark}

\numberwithin{equation}{section}

\makeatletter
\@namedef{subjclassname@2020}{\textup{2020} Mathematical Subject Classification}
\makeatother
\title{A model structure and Hopf-cyclic theory on the category of coequivariant modules over a comodule algebra} 
\date{\today} 
\author{Mariko Ohara}
\address{Center for Liberal Arts and Sciences, Faculty of Engineering, Toyama Prefectural University, 5180, Kurokawa, Imizu City, Toyama Prefecture, JAPAN, 939-0398. 
\\ Tel : +81-766-56-7500.
}  
\email{primarydecomposition@gmail.com}
\subjclass[2020]{ 16T15, 18G20, 57T05 55N25 (primary),  16E30, 55U10 (secondary)}
\keywords{Hopf algebra, derived category, model category}

\newcommand{\Ker}{{\mathrm{Ker}}}

\newcommand{\Hom}{\mathrm{Hom}}

\newcommand{\lmod}{\underline{\mathrm{LMod}}}

\newcommand{\LMod}{\mathrm{LMod}}

\newcommand{\rmod}{\underline{\mathrm{RMod}}}

\usepackage{enumerate}
\usepackage{array}
\usepackage[all]{xy} 
\usepackage{delarray}
\ifx\undefined\bysame
\newcommand{\bysame}{\leavemode\hbox to3em{\hrulefill}\,}
\fi
\begin{document}
\thispagestyle{empty}

\begin{abstract}
Let $H$ be a coFrobenius Hopf algebra over a field $k$. Let $A$ be a right $H$-comodule algebra over $k$. 

We recall that the category $\mathcal{M}^H$ of right $H$-comodules admits a certain model structure whose homotopy category is equivalent to the stable category of right $H$-comodules given in \cite{Fari}. In the first part of this paper, 
we show that the category $\LMod_A(\mathcal{M}^H)$ of left $A$-module objects in $\mathcal{M}^H$ admits a model structure, which becomes a model subcategory of the category of $A \# H^\ast$-modules endowed with a model structure given in \cite{Ohara} if $H$ is finite dimensional with a certain assumption. Note that $\LMod_A(\mathcal{M}^H)$ is not a Frobenius category in general. We also construct a functorial cofibrant replacement by proceeding the similar argument as in \cite{Qi}. 

Hopf-cyclic theory is refered as a theory of cyclic homology of (co)module (co)algebra over a Hopf algebra $H$ whose coefficients in Hopf $H$-modules. 
In the latter half of this paper, we see that cyclic $H$-comodules which give Hopf-cyclic (co)homology with coefficients in Hopf $H$-modules are contructible in the homotopy category of right $H$-comodules, 
and we investigate a Hopf-cyclic (co)homology in slightly modified setting by assuming $A$ a right $H$-comodule $k$-Hopf algebra with $H$-colinear bijective antipode in stable category of right $H$-comodules and give an analogue of the characteristic map.  
\end{abstract}

\maketitle

\section{Introduction}
Hopf algebras play the central role of representation theory and theory of decategorification. Recently, Khovanov~\cite{Khov}, Qi~\cite{Qi} and Farinati~\cite{Fari} defined the derived category of $H$-modules, $A \# H$-modules for an $H$-module algebra $A$ and $H$-comodules, respectively. They calculated the Grothendieck groups $K_0$ and $G_0$ of the derived categories of $A \# H$-modules and right $H$-comodules, respectively. 

In \cite{Ohara}, the author showed that the category of $A \# H$-modules admits a certain model structure under certain assumption, and the derived category of $A \# H$-modules in \cite{Qi} is arising from the model structure.


For $H$-comodules, Hovey gave a model structure on the category of chain complexes on $H$-comodules. 
Hess and Shipley generalized the stable model structure on the category of chain complexes of $H$-modules as a left induced model structure and related with the model structure on the category of simplicial sets and spectra, respectively.  
Hess, K\c{e}dziorek, Riehl and Shipley observed a necessary and sufficient condition for induced model structure which is combinatorial and studied the category of differential graded comodules.

The category of modules over a Frobenius algebra admits a model structure whose weak equivalences are the stable equivalences by Hovey~\cite{Hov}. Li generalized the model structure to an exact Frobenius category. The model structure is well-accepted among experts but the proof is written in his paper for the first time~\cite{Li}.

A Hopf algebra $H$, such that its linear dual $H^\ast$ has a left integral, is called coFrobenius as in Definition~\ref{cofrob}. 
Let $H$ be a coFrobenius Hopf algebra. Let $A$ be a right $H$-comodule algebra in the category $\mathcal{M}^H$ of right $H$-comodules. 
In this paper, we recall that the category of left $A$-module objects in $\mathcal{M}^H$ is endowed with the stable model structure. 
Here $\lambda : H \to k$ is a cointegral, $S$ the antipode and $\eta$ the unit map.  
In the case of $A \# H$-module as in \cite{Qi} and \cite{Ohara}, $\Hom_A(-, -)$ is naturally an $H$-module. However, in the case of $H$-coequivariant $A$-modules, $\Hom_A(-, -)$ is not an $H$-comodule in general. We also modified suspention and desuspention defined in \cite{Fari} to be compatible with $A$-module structure.  

We have the following. 
\begin{thm}[Proposition~\ref{main}]
Let $H$ be a coFrobenius Hopf algebra and $\mathcal{M}^H$ the category of right $H$-comodules together with $H$-colinear maps endowed with the stable model structure. Let $\LMod_A(\mathcal{M}^H)$ be the category of left $A$-modules of $\mathcal{M}^H$ for an $H$-comodule algebra $A$. Let us denote the forgetful functor by $U : \LMod_A(\mathcal{M}^H) \to \mathcal{M}^H$. 

Then, $\LMod_A(\mathcal{M}^H)$ admits a model structure with respect to the three classes of maps; a map $f$ in $\LMod_A(\mathcal{M}^H)$ is a weak equivalence if $U(f)$ is a weak equivalence, a fibration if $U(f)$ is a fibration and a cofibration if $f$ has the left lifting property with respect to trivial fibrations, respectively. 
\end{thm}

We remark that the model structure in the main theorem and the right-induced model structure in the paper of Hess and Shipley~\cite{HeS1} are introduced in different situation, respectively, but the model structure on the category of left $A$-module objects in $\mathcal{M}^H$ in this paper has the similar structure of right induced model structure~\cite[Definition 4.1]{HeS1}.

If $H$ is finite dimensional, we have a comparison with the model structure in \cite{Ohara} via the functor $i$ given in \cite{Hov}. 

\begin{cor}
Let $H$ be a finite dimensional coFrobenius commutative Hopf algebra and $\mathcal{M}_{H^\ast}$ the category of left $H^\ast$-modules. 
Let us take a right $H$-comodule algebra $A$, which can be regarded as a left $A \# H^\ast$-module. Let $i : \mathcal{M}^H \to \mathcal{M}_H^{\ast}$ be the inclusion functor and $\tilde{i}$ the induced functor on the category of left $A$-modules. Then, we have the following commutative diagram of functors. 
\[ 
 \xymatrix@1{
\LMod_A(\mathcal{M}^H) \ar[d]_{U}  \ar[r]^{\tilde{i}} & \LMod_A(\mathcal{M}_{H^\ast}) \ar[d]^{U} \\
\mathcal{M}^H \ar[r]^{i} & \mathcal{M}_{H^\ast} 
} \] 
Here, we denote by $U$ the forgetful functors. 

Assume that $H$ is coFrobenius and that $\LMod_A(\mathcal{M}_{H^\ast})$ inherits a model structure which is defined as in \cite{Ohara}. Then, $\tilde{i}$ and $i$ are left and right Quillen functor and $U$ are right Quillen functors.  
\end{cor}

Let $k$ be a field and $A$ a $k$-algebra. Let $C_{\ast}(A)$ the Hochschild chain complex, with $C(A)_n=A^{\otimes n+1}$. 
This induces a bicomplex $B(A)_{\ast \ast}$ with $B(A)_{pq}=A^{\otimes (q-p+1)}$ if $q \ge p$ and $0$ otherwise, whose the vertical differential $b$ is the Hochschild differential and the horizontal differential is given by the Connes' operator $B_n : A^{\otimes n+1} \to A^{\otimes n+2}$. The cyclic homology of $A$ is defined by $HC_{\ast}(A)=H_{\ast}(Tot B_{\ast \ast}(A))$. This definition makes sense for a mixed complex $(M , b, B)$, where $M$ is a graded $k$-module with differentials $b$ and $B$ such that $bB+Bb=0$. 
An alternative definition of cyclic homology is $HC_{\ast}(M)=H_{\ast}(k \otimes_{\Lambda}^{\mathbb{L}}M)$ for a mixed complex $(M, b, B)$, where $(-) \otimes_{\Lambda}^{\mathbb{L}}(-)$ stands for the derived tensor product.

Let us define the notation $\Lambda (x, y)$ to be the signed Hopf algebra $k\{x, y \} / (x^2 , y^2, xy+yx)$. 
In \cite{Fari}, the Hopf algebra $k[\mathbb{Z}] \# \Lambda (x, y)$ is defined to be $k\{ g^{\pm 1}, x, y \} / (gx=-xg, gy=-yg, 0=x^2=y^2=xy+yx)$ together with the comultiplication map given by $\Delta (g)=g \otimes g$, $\Delta (x)=x \otimes g + 1 \otimes x$ and $\Delta (y)=y \otimes g^{-1}+1 \otimes y$. In this case, $x$ and $y$ are corresponding to the differential of degree $+1$ and $-1$, respectively. 
He also calculated the Hopfological cohomology of a mixed complexes, which is in turn slightly different from the cyclic homology~\cite[Section 5.3]{Fari}.

Hopf-cyclic homology is introduced by Connes and Moscovich~\cite{Hopf1}. It has been studied by Kaygun and Khalkhali~\cite{KK}, Rangipour and many other researchers and resently generalized to certain Hopf algebroids by Balodi, which can be refered as a theory of cyclic homology of (co)module (co)algebra over a Hopf algebra $H$ with coefficients in stable Hopf $H$-modules. 
For example, Hopf-cyclic theory appears in Hopf-Galois extension. For a right $H$-comodule algebra $A$ and the coinvariant subalgebra $B=A^{coH}$, the extension $B \subset A$ is called Galois if the natural map $A \otimes_B A \to A \otimes H$ is bijective. 
By the result of Jara and \c{S}tefan, the cyclic structure on the relative cyclic complex of $B \subset A$ gives rise to a cyclic module via the natural map, depending only on $H$ and stable anti-Yetter Drinfeld module $A/ [A, B]$. 

We recall the properties when a relative $(A, H)$-Hopf module becomes a projective $A$-module and the Hopf-cyclic modules with a stable $H$-module / comodule is equivalent to zero in the homotopy category of right $H$-comodules. Here, a stable $H$-module / comodule is an $H$-comodule wich is also $H$-module together with $m \circ \Delta = id_H$. 
So we would like to consider a slightly modified cyclic homology assuming $A$ a Hopf algebra $A$, induce an analogue of the characteristic map and compare it with one in the category of chain complexes of $H$-comodules.   
\begin{thm}[Corollary~\ref{main2}]
The Hopf-cyclic homology of $H$-comodule $k$-Hopf algebra $A$ together with bijective $H$-colinear antipode $S: A \to A$ whose coefficients in stable $A$-module / comodule in the category of $H$-comodules satisfies the universal approximation as in \cite{Kay}. We also obtain analogue of Hopf-cyclic and equivariant Hopf-cyclic theory. 
There is a map $Q^A_\bullet(A, k) \to Q^A_\bullet(A, A)$ of right $H$-comodules in $\mathcal{M}^H$ induced from the unit map. 
\end{thm}

We remark that, as an expansion of an idea of taking $k$ as in $Q^A_\bullet(A, k)$, if we take an $A$-coinvariant part of $M$ assuming that $M$ is a Hopf $A$-module in $\mathcal{M}^H$, we have the degree shift of $Q_\bullet^A(A, M)$ as in Remark~\ref{shift}.

\subsubsection*{Acknowledgement}
The author would like to thank Professor Dai Tamaki to suggest studying Hopf algebras and a lot of comments over the topic of Hopf-cyclic theory during the author's stay in Shinshu University. The author also would like to thank Professor Takeshi Torii to advise comparing right $H$-comodules with the category of certain dg-modues endowed with the several model structure.

\section{The category of $H$-comodules and left module objects over a monoid $A$}
Let $k$ be a field. 
Let $H$ be a Hopf algebra over $k$. 
We write $\otimes$ for $\otimes_k$. Throughout this paper, $H$ will be  assumed to be coFrobenius. The condition is given in Definition~\ref{cofrob}. 

\subsection{The category of right $H$-comodules}

A right $H$-comodule is a $k$-vector space $M$ together with a linear map $\rho : M \to M \otimes H$ satisfying two conditions $(id \otimes \Delta) \circ \rho = (\rho \otimes id) \circ \rho$ and $(id \otimes \varepsilon) \circ \rho = id$. Here $\Delta$ and $\varepsilon$ are the comultiplication map and the counit map of $H$, respectively. 

The following lemma is a basic fact (cf. \cite[Lemma 3.14]{Fari}). 
\begin{lemma}\label{314}
For any right $H$-comodule $M$, let $V$ be the underlying $k$-vector space of $M$. Then, we have an isomorphism $M \otimes H \cong V \otimes H$ of right $H$-comodules given by $m \otimes h \mapsto m_{(0)} \otimes m_{(1)}h$. The inverse map is given by $m \otimes h \mapsto m_{(0)} \otimes S(m_{(1)})h$, which are coming to a right $H$-comodule map. Here, we regard $M \otimes H$ and $V \otimes H$ as right $H$-comodules via the diagonal map and the map $id \otimes \Delta$, respectively. 
\end{lemma}


If $M$ is a right $H$-comodule, then $M$ can be regarded as a left $H^\ast$-module. Here $H^\ast$ stands for the linear dual $\Hom_k(H, k)$ of $H$. A left $H^\ast$-module is not necessarily a right $H$-comodule in general.

The right $H$-comodule structure of $M \otimes_k N$ consisting of two right $H$-comodules $M$ and $N$ is given by diagonal; $(id \otimes m) \circ (id \otimes \tau \otimes id) \circ (\rho_M \otimes \rho_N) :  M \otimes_k N \to M \otimes_k H \otimes_k N \otimes_k H \to M \otimes_k N \otimes H \otimes_k H \to M \otimes_k N \otimes_k H$. Here, $\rho_M$ and $\rho_N$ is the right $H$-comodule structure of $M$ and $N$, respectively, and $\tau$ is the braiding. 
As a special case, we regard a right $H$-comodule of the form $M \otimes H$ as a right $H$-comodule via the structure map $id \otimes \Delta$. By Lemma~\ref{314}, it is isomorphc as right $H$-comodules to $M \otimes H$ via the diagonal coaction.  

We say a map between right $H$-comodules an $H$-colinear map if it is compatible with right $H$-comodule structures. 
We denote by $\Hom^H(M, N)$ the morphism space of maps of $H$-colinear right $H$-comodules from $M$ to $N$.   

Let us denote the category of right $H$-comodules by $\mathcal{M}^H$. 

\begin{defn}[\cite{AC2013}, \cite{Fari}]\label{cofrob}
Let $k$ be a field and $H$ a Hopf algebra over $k$. A left integral $\Lambda : H \to k$ is a map satisfying $(id \otimes \Lambda) \Delta h=\Lambda (h) 1$ for any $h \in H$. In other words, we have $h_{(1)}\Lambda (h_{(2)})=\Lambda (h) 1$. For a Hopf algebra $H$, if its linear dual $H^{\ast}=\Hom_k(H, k)$ admits a left integral $\Lambda$, we refer to $H$ as a coFrobenius Hopf algebra. 
\end{defn}

\begin{rem}
A Hopf algebra $H$ being right coFrobenius amounts to being left coFrobenius, so that  it is left-right symmetric for Hopf algebras. Note that it is not the case for coalgebras. 
\end{rem}


Throughout this paper, we assume that $H$ is coFrobenius.

Let $H$ be a coFrobenius Hopf algebra. 
If $M$ is a right $H$-comodule, we can think of $M$ as a left $H^{\ast}$-module via the structure map $M \otimes H^{\ast} \to M \otimes H \otimes H^\ast \to M$, where the last map is the evaluation map $H \otimes H^\ast \to k$. This defines fully faithful functor from the category of right $H$-comodules to the category of left $H^{\ast}$-modules~\cite{Sweedler}. 
We say that a $H^{\ast}$-module $M$ is rational (or tame) if, for all $m \in M$, the submodule generated by $m$ is finite-dimensional. 
The essential image of the fully faithful functor consists of rational $H^\ast$-modules. 


If $H$ is coFrobenius, then every finite dimensional $H$-comodule is a quotient of a finite dimensional projective $H$-comodule and has an embedding to a finite dimensional injective $H$-comodule. Especially, the category of right $H$-comodules has enough projectives and enough injectives. Also the category $\mathcal{M}^H$ of right $H$-comodules becomes a Frobenius category~\cite[Theorem 2.8]{AC2013}, \cite[Theorem 2.2]{Fari}. 

We know that any injective right $H$-comodule is the direct summand of the form $M \otimes H$~\cite[Lemma 3.13 (2)]{Fari}. 

Let $H$ be a coFrobenius Hopf algebra over a field $k$. The category of right $H$-comodules has enough projectives if and only if $H$ is semiperfect, i.e., every simple $H$-comodule has an injective hull which is finite-dimensional as an $H$-vector space. This result is attributed to \cite{Lin}. 

The category of right $H$-comodules has all small limits and colimits and is locally presentable~\cite{Porst}. 

\begin{defn}[Stable equivalence]
Let $X$ and $Y$ be two $H$-comodules in $\mathcal{M}^H$. A map $f, g : X \to Y$ in $\mathcal{M}^H$ is a stably equivalent if $f-g$ factors through an injective $H$-comodule. This is an equivalence relation that is compatible with composition. A map $f, g : X \to Y$ of $H$-comodules is a stable equivalence if it is an isomorphism after taking stable equivalence classes of right $H$-comodule maps.  
\end{defn}


\begin{defn}
Let $M$ and $N$ be right $H$-comodules. Let $\Hom^H(M, N)_0$ be a $k$-submodule of $\Hom^H(M, N)$ consisting of those maps which factors through an injective right $H$-comodule. 
A set of maps $\underline{\Hom^H(M, N)}$ is defined to be $\Hom^H(M, N) / \Hom^H(M, N)_0$.  
 \end{defn}

Let $C$ be an additive full subcategory of an abelian category which is extension closed and $E$ is a class of short exact sequences in $C$. A map $A \to B$ of a short exact sequence $A \to B \to C$ in $E$ is called a cofibration and $B \to C$ called a fibration. Since the category of $H$-comodules becomes a Frobenius category, e.g., the class of enough injectives coincides with the class of enough projectives. Therefore the category of $H$-comodules becomes a Frobenius category. From \cite[Theorem 1.1]{Li}, an abelian Frobenius category admits the model structure given by the following three classes of maps; a cofibration is a monomorphism, a fibration is a epimorphism and a weak equivalence is a stable equivalence.

\begin{prop}[\cite{Li}, Theorem 1.1]\label{mh}
For a coFrobenius Hopf algebra $H$ over a field $k$, 
the category of right $H$-comodules becomes an abelian Frobenius category, so that it inherits the stable model structure, e.g., the class of cofibrations consisting of monomorphisms, fibrations consisting of epimorphisms and weak equivalences consisting of stable equivalences. 
\end{prop}
\qed

We denote by $\underline{\mathcal{M}^H}$ the homotopy category with respect to the stable model structure. 
Since all objects are cofibrant, the category $\underline{\mathcal{M}^H}$ is categorically equivalent to the homotopy category defined in \cite{Fari}. Moreover, we can construct the derived category of $\mathcal{M}^H$, by inverting weak equivalences, which is arising from the model structure and is categorically equivalent to the derived category given in \cite{Fari}.

Next, we recall suspention and desuspention of a right $H$-comodule as in \cite{Fari}.  
Let $M$ be a right $H$-comodule with the structure map $\rho$. Let $M \otimes_k H$ be a right $H$-comodule with the structure map $id \otimes_k \Delta$. By definition of a right $H$-comodule, we have $(id \otimes \Delta)\rho = (\rho \otimes id)\rho$ for a structure map $\rho : M \to M \otimes_k H$ of $M$. Then, the suspension of $M$ is defined by $T(M)=(M \otimes H) / \rho (M)$, where the right $H$-comodule structure of $M \otimes H$ is given by $id \otimes \Delta$. 
On the other hand, we have $(\Lambda' \otimes id)\Delta (h)= \Lambda(S(h_{(1)}))h_{(2)}=\Lambda (S(h_{(1)}))S^{-1}S(h_{(2)})=S^{-1}(\Lambda(S(h_{(1)})S(h_{(2)})))=S^{-1}(\Lambda(Sh)1)=\Lambda'(h)1$, so that $\Lambda'$ is compatible with the structure map $\Delta$ as a right $H$-comodule of $H$ and the trivial $H$-coaction on $k$. 

For a left integral $\Lambda : H \to k$, $\Lambda' = \Lambda \circ S$ is right integral, so that $\Ker (\Lambda')$ becomes a right $H$-comodule. The desuspension of $M$ is defined by $T^{-1}(M)=M \otimes \Ker (\Lambda')$. 

By Lemma \ref{314}, we have an $H$-colinear isomorphism on $M \otimes_k H$ between the right $H$-comodule structure via diagonal coaction and via $id \otimes \Delta$. 

By definition, we have the following exaxt sequences; 
\[ 0 \to M \to M \otimes H \to T(M) \to 0 \]
and
\[ 0 \to T^{-1}(M) \to M \otimes H \to M \to 0 .  
\]

We replace $M$ as $T'M$ in the first short exact sequence, then we have $TT' \simeq id$ and $T'T \simeq id$ in the homotopy category. 
Thus, the homotopy category $\underline{\mathcal{M}^{H}}$ and the derived category becomes triangulated. 

These suspention $T$ and desuspention $T^{-1}$ are useful when chasing exact sequence of cohomology as in \cite{Fari}. 
However, we will define a slightly different suspention and desuspention as following definition since we will realize them as left $A$-module objects for a right $H$-comodule algebra $A$ later in section $2.2$. 

\begin{defn}\label{choice}
Let $M$ be a right $H$-comodule. We will choice an injective embedding as follows. 
\begin{enumerate}[(i)]
\item First, since $H$ is coFrobenius, $\Lambda'$ is not a zero map, so that there exists 
$x \in H$ such that $\Lambda'(x)$ is not zero.
\item There exists finite dimensional $H$-subcomodule in $H$ that contains $x$. (we can take $x=\Sigma \epsilon (x_0)x_1$, where $x_1$ is regarded as a basis of $k$-vector space). We take one of these $k$-basis $x_1$s whose image under $\Lambda'$ is not zero. We write it as $x^1$.
\item If we take another $x^2$ as (ii), we see that $(x^1) \otimes H \cong (x^2) \otimes H$ as $H$-comodules since we have an $H$-comodule isomorphism $M \otimes H \cong V \otimes H$ as Lemma~\ref{314}, where $V$ is the underlying $k$-module. Thus we have $kx^1 \otimes H \cong kx^2 \otimes H$. Here, $kx^1$ and $kx^2$ are both $k$-vector space generated by $x^1$ and $x^2$, respectively
\item Then, we take an injective embedding: $ id \otimes x^1 : M \to M \otimes H$ : the right component is always $x^1$ of (ii). 
\end{enumerate}
\end{defn}

\begin{rem}[The existence of a functorial injective embedding]\label{choice}
For the trivial module $k$, if we choose a monomorphism $k \to B$ where $B$ is an injective $H$-comodule, by tensoring with this map, we obtain an injective embedding $M\to M\otimes B$ instead of $M\to M\otimes H$. 
When $H$ is coFrobenius, we can choose those injective $B$ to be finite dimensional. 
\end{rem}

\begin{defn}\label{sus}
We define $\Sigma (M)$ by the pushout diagram in $\mathcal{M}^H$
\[ 
 \xymatrix@1{
M \ar[d]_{id \otimes x}  \ar[r] & 0 \ar[d] \\ 
M \otimes H \ar[r] & \Sigma(M) 
} \] 
Here, $id \otimes x$ is an injective embedding we take as in Definition~\ref{choice}. Note that $\Sigma(M)$ is equivalent to $T(M)$ in $\underline{\mathcal{M}^{H}}$. 

As in \cite{Fari}, we define $\Sigma^{-1}(M)$ by the pullback in $\mathcal{M}^H$
\[ 
 \xymatrix@1{
 \Sigma^{-1}(M) \ar[r] \ar[d]  & 0 \ar[d]\\
M \otimes H  \ar[r]_{id \otimes \Lambda'} & M  
} \] 

\end{defn}

\subsection{The category of left modules over a monoid object $A$ in $\mathcal{M}^H$}
We say that a right $H$-comodule $A$ is an $H$-comodule algebra $A$ if it is a $k$-algebra $A$ together with a right $H$-comodule structure such that the multiplication map $A \otimes A \to A$ and the unit $k \to A$ are $H$-colinear, i.e., compatible with $H$-comodule structure. 

Let $\LMod_A(\mathcal{M}^H)$ be the category of left $A$-module objects in $\mathcal{M}^H$. We say an object $M$ in $\LMod_A(\mathcal{M}^H)$ an $H$-coequivariant $A$-module.

For an $H$-coequivariant $A$-module $M$ and right $H$-comodule $V$, we have an $H$-coequivariant $A$-module $M \otimes V$, 
where $A$-module structure in $M \otimes V$ is the one coming from $M$ and the right $H$-comodule structure is the diagonal one. 

As in \cite{Khov} and \cite{Qi}, we consider the restriction functor 
\[
\LMod_A(\mathcal{M}^H) \to \mathcal{M}^H
\]
and define $M$ in $\LMod_A(\mathcal{M}^H)$ to be contructible if $M$ is injective as right $H$-comodule. We say that a map $f : M \to N$ in $\LMod_A(\mathcal{M}^H)$ is a stable equivalence if it becomes an isomorphism in $\underline{\mathcal{M}^H}$. 




For $H$-coequivariant $A$-modules $M$ and $N$, we denote by $\Hom_A^H(M, N)$ the set of maps that are $H$-colinear and also $A$-linear. 

If $H$ is a coFrobenius Hopf algebras, for an $H$-comodule algebra $A$ and $H$-coequivariant $A$-modules $M$ and $N$, $\Hom_A(M, N)$ does not have an $H$-comodule structure in general~\cite[Section 7.1]{Fari}. 
Therefore, we regard right $H$-comodules $M$ and $N$ as left $H^\ast$-modules via the evaluation map $H \otimes H^\ast \to k$, 
we obtain a left $A$-module left $H^\ast$-module $M$. 

\begin{defn}\label{actf}
An $H^\ast$-action of $x \in H^\ast$ on $f : M \to N$ is given by $xf=(1 \otimes m^\ast_H(x))(\rho_N \otimes 1)(f \otimes S)\rho_M$. In other words, 
\[
(xf)(m)=x(f(m_{(0)})_{1}S(m_{(1)}))f(m_{(0)})_0. 
\] 

Then, $f$ is $H$-colinear if and only if it is $H^\ast$-linear, and  $f$ is $H$-colinear if and only if $xf=\epsilon (x) f = x(1) f$ for any $x \in H^\ast$~\cite[Proposition 7.6]{Fari}. 
\end{defn}


The set $\Hom_A$ itself may not be an $H$-comodule but admits $H^\ast$-action. So we can take "invariant of a  morphism space of $H^\ast$-linear maps". 

\begin{prop}
\label{equiv}
$\Hom_A(M, N)$ is an $H^{\ast}$-module. 
Explicitly, for an $A$-module map $f : M \to N$ between $H$-coequivariant $A$-modules, it is $H^\ast$-linear if and only if it is $H$-colinear. 

Moreover, any map $f \in \Hom_A^H(M, N)$ between $H$-coequivariant $A$-modules $M$ and $N$ can be regarded as a map in $\Hom_A(M, N)$ whose $H^\ast$-action as in Definition~\ref{actf} is the action given by the counit $\epsilon \in H^\ast$. 
\end{prop}
Note that any map $f \in \Hom_A(M, N)$ inherits an $H^\ast$-action as in Definition~\ref{actf}. 
\begin{proof}
The first statement is well-known as in \cite{Qi} and followed by the fact that $f$ is $H$-colinear if and only if $xf=\epsilon (x) f = x(1) f$ for any $x \in H^\ast$. 

For the second statement, if $L$ is an $H^\ast$-module, one may define 
$L^{H^\ast}=\{ m \in L | xm=x(1)m \} \cong \Hom_{H^\ast}(k, L)$. 
If $L$ stands for $\Hom_A(M, N)$, then $L^{H^\ast}$ consists of $A$-linear and $H^\ast$-linear maps. 
\end{proof}
Note that if $L$ is a right $H$-comodule, we have $L^{coH}=\{ m \in L | \rho (m)=m \otimes 1 \}=L^{H^{\ast}}$.

If $M$ and $N$ are objects in $\LMod_A(\mathcal{M}^H)$, we denote by $I_A(M, N)$ the subset of $\Hom^H_A(M, N)$ consisting of maps that factors through an $H$-coequivariant $A$-modules that are injective as right $H$-comodules. 
The stable homotopy category denoted by $\lmod_A(\mathcal{M}^H)$ is defined as the category with same objects as $\LMod_A(\mathcal{M}^H)$ but morphism 
\[
\Hom_{\lmod_A(\mathcal{M}^H)}(M, N)= \Hom_A^H(M, N) / I_A(M, N) . 
\]

The derived category $D_H(A)$ is defined by formally inverting stable equivalences in the homotopy category $\rmod_A(\mathcal{M}^H)$.

Especially, the projective objects are preserved via the exact functor
$$\underline{\mathrm{U}} : \lmod_A(\mathcal{M}^H) \to \underline{\mathcal{M}^H}$$ induced from the forgetful functor. 

The suspention $\Sigma (M)$ and desuspention $\Sigma^{-1}(M)$ in $\mathcal{M}^H$ actually gives objects in $\LMod_A(\mathcal{M}^H)$. 

From now on, for a right $H$-comodule map $\Lambda' : H \to k$, that is the composition of an integral $\Lambda : H \to k$ with the antipode $S : H \to H$. 


%

\begin{rem}\label{cmdstr}
Let $X$ be an $H$-coequivariant $A$-module. In \cite{Fari}, the structure map $\rho : X \to X \otimes H$ is regarded as an injective embedding. 
If we regard the field $k$ as the trivial $H$-comodule via $\eta$ and the Hopf algebra $H$ as the right $H$-comodule via the comultiplication $\Delta$, the map $\eta$ is coming to be a right $H$-comodule map. 

Note that $id_X \otimes \eta : X \to X \otimes H$, $id_X \otimes \varepsilon : X \otimes H \to X$ and the identity map on $X$ are left $A$-module maps but $id \otimes \varepsilon$ is not a right $H$-comodule map. Here $X \otimes H$ is endowed with a right $H$-comodule structure by $id \otimes \Delta$. 


Assume that $X \otimes H$ is a right $H$-comodule by $id \otimes \Delta$. 
Let us consider the following exact sequences on $k$-vector space level. 
\[ 
 \xymatrix@1{
X  \ar[r]^-{\rho} & X \otimes H \ar[r]^-{id_X \otimes \varepsilon} & X .
} \] 
Note that $\rho$ is a map of right $H$-comodules. 
Since $H$ is a bialgebra, this is a split exact sequence on $k$-vector space level. 
On the other hand, we have an exact sequence of $k$-vector space 
\[ 
 \xymatrix@1{
X  \ar[r]^-{id_X \otimes \eta} & X \otimes H \ar[r]^-{id_X \otimes \varepsilon} & X .
} \] 
This is also a split exact sequence on $k$-vector space level and also splits as an exact sequence of left $A$-modules when we leave the $H$-comodule structures. Note that $A$ acts only on each left component and $id \otimes \eta$ is a map of $H$-coequivariant left $A$-modules. 


\end{rem}

We define a mapping cylinder and a mapping cocylinder.  
\begin{defn}
We define an $H$-coequivariant $A$-module $C_f$ by the pushout
\[ 
 \xymatrix@1{
X \ar[r]^f \ar[d]_{id \otimes x}
 & Y \ar[d]  \\
X \otimes H \ar[r] & C_f .
} \] 
Here, we regard $X \otimes H$ as a right $H$-comodule via the structure map $id \otimes \Delta$. 

As the following diagram and Definition~\ref{sus}, we obtain the map $C_f  \to \Sigma X$ by the universality. 
\[ 
 \xymatrix@1{
X \ar[r]^f \ar[d]_{id \otimes x} & Y \ar[d] \ar[r] & 0 \ar[dd]  \\
X \otimes H \ar[r] \ar[d]_{id} & C_f  \ar@{.>}[dr] &\\
X \otimes H \ar[rr] & & \Sigma X 
} \] 

We have the commutative diagram
\[ 
 \xymatrix@1{
X \ar[r]^f \ar[d]_{id_X \otimes x} & Y \ar[d] \ar[rdd]^{id_Y} & \\
X \otimes H \ar[r] \ar[drr]_{f \otimes \Lambda'} & C_f  \ar@{.>}[dr] &\\
 & & Y \otimes k \cong Y .  
} \] 

Similarly, we define an $H$-coequivariant $A$-module $P_f$ by the pullback

\[ 
 \xymatrix@1{
P_f \ar[r] \ar[d]& X \ar[d]^f  \\
Y \otimes H \ar[r]_{id_Y \otimes \Lambda'} & Y \otimes_k k \cong Y.
} \] 


As the following diagram and Definition~\ref{sus}, we obtain the map $\Sigma^{-1}Y \to P_f$ by the universality. 
\[ 
 \xymatrix@1{
 \Sigma^{-1}Y \ar@{.>}[dr] \ar[dd] \ar[rr] & & 0 \ar[d]\\
  & P_f \ar[r] \ar[d]& X \ar[d]^f  \\
Y \otimes H \ar[r]_{id} & Y \otimes H \ar[r]_{id_Y \otimes \Lambda'} & Y  
} \] 

We have the commutative diagram
\[ 
 \xymatrix@1{
X \ar[drr]^{id_X} \ar[ddr]_{f \otimes 1} \ar@{.>}[dr] & & \\
  & P_f \ar[r] \ar[d]& X \ar[d]^f  \\
& Y \otimes H \ar[r]_{id_Y \otimes \Lambda'} & Y  .
} \] 
\end{defn} 

As a variant of \cite[Lemma 4.3]{Qi} and also \cite[Lemma 2.10]{Fari}, we have the following lemma of standard triangles. 
\begin{lemma}\label{Asplit}
For a map $f : X \to Y$ in $\LMod_A(\mathcal{M}^H)$, the maps $Y \to C_f$ and $P_f \to X$ are $A$-split. 

Especially, two short exact sequences 
\[
0 \to Y \to C_f \to \Sigma X \to 0
\]
and 
\[
0 \to \Sigma^{-1}Y \to P_f \to X \to 0
\]
are $A$-split. 

Moreover, the both two exact sequences induce a triangle $X \to Y \to C_f \to \Sigma X$ in the homotopy category $\lmod_A(\mathcal{M}^H)$.


\end{lemma}
\begin{proof}
Since $id_X \otimes x : X \to X \otimes H$ and $id_Y \otimes \Lambda' : Y \otimes H \to Y$ are $A$-split, the first assertion follows. 

The last claim follows since we have a short exact sequence $0 \to X \otimes H \to C_f \to \mathrm{Coker}(f) \to 0$, and so $C_f \simeq \mathrm{Coker}(f)$ in the homotopy category. 


\end{proof}

Thus, $\lmod_A(\mathcal{M}^H)$ is a triangulated category.

\section{A model structure on the category of $H$-coequivariant $A$-modules}
Now, we use the adjunction
\[
A \otimes_k (-) : \mathcal{M}^H \rightleftarrows \LMod_A(\mathcal{M}^H) : U , 
\]
where $U$ is the functor which regards $H$-coequivariant $A$-modules as $H$-comodules and $A \otimes_k (-)$ is the left adjoint, which is given by the free functor.

By using the stable model structure on $\mathcal{M}^H$ given in Definition~\ref{mh} and this adjunction, we take the following three classes of maps in the category 
$\LMod_A(\mathcal{M}^H)$. 

\begin{defn}\label{modelstr}
Let $f : X \to Y$ be a map in $\LMod_A(\mathcal{M}^H)$. 
\begin{enumerate}[(i)]
\item We say that $f$ is a weak equivalence if $Uf$ is a weak equivalence. 
\item We say that $f$ is a fibration if $Uf$ is a fibration. 
\item We say that $f$ is a cofibration if $f$ has the left lifting property with respect to trivial fibrations. 
\end{enumerate}
\end{defn}

We will see that the three classes of maps as in Definition~\ref{modelstr} defines a model structure on the category $\LMod_A(\mathcal{M}^H)$. Note that the category $\LMod_A(\mathcal{M}^H)$ is abelian so it is bicomplete. The proof will have many parallels with the case of $A \# H$-modules in \cite{Ohara}

\begin{lemma}[2 out of 3]
Let $f$, $g$ and $g \circ f$ be morphisms in $\LMod_A(\mathcal{M}^H)$. If two of the three  morphisms are weak equivalences in $\LMod_A(\mathcal{M}^H)$, then, so is the third. 
\end{lemma}
\begin{proof}
Since $U$ is a covaiant functor, the composition $U(f) \circ U(g)$ is $U(f \circ g)$. The assertions follow from restricting these maps on $H$-comodules under the functor $U$.  
\end{proof}
\begin{lemma}
Let $f$ and $g$ be maps of $\LMod_A(\mathcal{M}^H)$ such that $f$ is a retract of $g$, i.e., $f$ and $g$ satisfies the following commutative diagram. 
\[ 
 \xymatrix@1{
M \ar[d]_f \ar[r] & C \ar[r] \ar[d]_g & M \ar[d]_f \\
N \ar[r] & D \ar[r] & N 
} \] 
where the horizontal composites are identities. 

If $g$ is a weak equivalence, cofibration, or fibration, respectively, then so is the third. 
\end{lemma}
\begin{proof}
When $g$ is a weak equivalence or fibration, respectively, we can see that a retract $f$ of $g$ is also a weak equivalence or fibration, respectively, by regarding the diagram as a diagram of $H$-modules via the functor $U$.  
If the map $g$ is a cofibration, for any trivial fibration $X \to Y$, 
the map $g$ has the left lifting property as follows
\[ 
 \xymatrix@1{
M \ar[r] \ar[d]_f & C \ar[d]_g \ar[r]& M \ar[d]_f \ar[r] & X \ar[d] \\
N \ar[r] &D \ar[r] \ar@{.>}[urr] & N \ar[r] & Y ,
} \] 
so we have  the left lifting of $f : M \to N$ with respect to $X \to Y$. 
\end{proof}

The lifting propeties for trivial fibrations are obvious and, for pairs of trivial cofibrations and fibrations, notice that the cokernel of a trivial cofibration as $H$-coequivariant $A$-module is an injective $H$-comodule. We obtain the certain form of direct sum of $H$-comodules and obtain the desired $H$-coequivariant $A$-module map by Proposition~\ref{equiv}. 
\begin{prop}[Factorization]
For any map $f : X \to Y$, we have a factorization $f : X \to E \to Y$ for some $E$, where $X \to E$ is a cofibration and $E \to Y$ is a trivial fibration. Also we have a factorization $f : X \to E' \to Y$, where $X \to E'$ is a trivial cofibration, and $E' \to Y$ is a fibration. 
\end{prop}
\begin{proof}
It suffices to show that, if we have a cofibrant replacement and fibrant replacement, respectively, we obtain the desired factorization of this lemma as follows. A cofibrant replacement is constructed in the following Proposition~\ref{cofrep} and all objects are fibrant.

For a cofibrant replacement $QC_f \to C_f$ of $C_f$, where the map $QC_f \to C_f$ is a trivial fibration, we have the following diagram : 
\[ 
 \xymatrix@1{
X \ar@{.>}[r] \ar[d]_{=} & P_u \ar@{.>}[d] \ar[r] & QC_f \ar[d]_{\simeq} \ar[r]^u & \Sigma X \ar[d]_{=} \\
X \ar[r]^f & Y \ar[r] & C_f \ar[r] & \Sigma X .
} \] 
Two vertial maps in the middle square is surjective and two horizontal maps in the middle square is $A$-split. We define the map $u$ as a map the right square commutes. We construct the map $P_u \to Y$ by using $A$-splittings, Lemma~\ref{Asplit} and Proposition~\ref{equiv}. We have the map $\Sigma^{-1} \Sigma X \to P_u$ by the universality of desuspension. The map $P_u \to Y$ is surjective since it is composed of two surjections. Passing to the quotient triangulated category, the map $P_u \to Y$ is a weak equivalence. 

The map $X \to P_u$ is a cofibration. 
Consider the following diagram
\[ 
 \xymatrix@1{
X  \ar[d] \ar[r] & M \ar[d] \\
P_u \ar[d] \ar[r] & N  \\
QC_f \ar@{.>}[uur] &  
} \] 
where $M \to N$ is a trivial fibration and $P_u \to QC_f$ is $A$-split. Since $QC_f$ is cofibrant, we have the dotted arrow. By using $A$-splitting endowed with trivial $H$-coaction, we have the lifting $P_u \to M$. 

For a fibrant replacement $P_f \to RP_f$, where the map $P_f \to RP_f$ is a trivial cofibration, we have the following diagram : 

\[ 
 \xymatrix@1{
\Sigma^{-1}Y \ar[r] \ar[d]_{=} & P_f \ar[d]_{\simeq} \ar[r] & X \ar@{.>}[d] \ar[r]^f & Y \ar[d]_{=} \\
\Sigma^{-1}Y \ar[r]^v & RP_f \ar[r] & C_v \ar@{.>}[r] & Y .
} \] 
Note that two horizontal maps in the middle square is $A$-split. The map $v$ is defined to be the left square commutes. We construct the map $X \to C_v$ by using $A$-splittings, Lemma~\ref{Asplit} and Proposition~\ref{equiv}. We have the map $C_v \to \Sigma \Sigma^{-1}Y$ by the universality of suspension. The map $X \to C_v$ is a stable equivalence since it is composed of two stable equivalences. 

Note that $X \to C_v$ is also cofibration since $A$-split monomorphisms are included in the class of cofibrations. 

The map $C_v \to Y$ is a fibration. 
Consider the following diagram
\[ 
 \xymatrix@1{
 & RP_f \ar[d] \\
M \ar[d] \ar[r] & C_v \ar[d] \\
N \ar[r]  \ar@{.>}[uur] &  Y
} \] 
where $M \to N$ is a trivial cofibration and $RP_f \to C_v$ is $A$-split. Since $RP_f$ is fibrant, we have the dotted arrow. By using $A$-splitting endowed with trivial $H$-coaction, we have the lifting $N \to C_v$. 
\end{proof}

\subsection{Functorial cofibrant replacement for $H$-coequivariant $A$-modules}

\begin{defn}\label{QiCof}
Let $M$ be an $H$-coequivariant $A$-module. 
$M$ is said to be cofibrant if for any surjective stable equivalence $L \to N$ of  $H$-coequivariant $A$-modules, the induced map of $k$-vector spaces
\[
\Hom_A^H(M, L) \to \Hom_A^H(M, N)
\]
is surjective. 

%
%
\end{defn}

As the similar argument in \cite{Qi}, we will construct a functorial cofibrant replacement. 
\begin{prop}[Functorial cofibrant replacement]\label{cofrep}
For each $H$-coequivariant $A$-module $M$, there is a short exact sequence in  $\LMod_A(\mathcal{M}^H)$ which is split exact as a sequence of $A$-modules,
$0 \to M \to aM \to \bar{p}M \to 0$
where $\bar{p}M$, and also $p(M)=\Sigma^{-1}(\bar{p}(M))$ is cofibrant and aM is an contructible and a left $A$-module. The construction of
the short exact sequence is functorial in $M$.
\end{prop}
\begin{proof}

This construction is an analogue of normalized chain complex of simplicial resolution. 
We construct the bar resolution for $M=A$. The proof also gives rise to the short exact sequence $0 \to M \to aM \to \bar{p}M \to 0$ claimed in the theorem. 

One can take such a simplicial module \{$C_\bullet$\} that the differential $\delta : C_n \to C_{n+1}$ is the alternative sum of the face maps $\delta = \Sigma_{0 \le i \le n-1}(-1)^id_i$. 
Then, $(C_{\bullet}, \delta)$ becomes a complex. 



We use an expedient notation $A \otimes H / (x)$ for the cokernel of $id \otimes x : A \to A \otimes H$. 
If $A$ is a right $H$-comodule algebra, all the face and degeneration maps are $H$-comodule maps. 
Therefore, $\delta = \Sigma_{0 \le i \le n-1}(-1)^id_i$  is also an $H$-colinear map~\cite[Example 2.8]{Qi}. 
We take the mapping cylinder of $\delta_0 : A \otimes A \to A$, and we have $C_{\delta_0} \cong A \otimes A \otimes (H/ (x)) \oplus A$ as left $A$-modules, so that we have an $A$-split filtration $\{0\} \subset A \subset C_{\delta_0}$, the isomorphism viewed as a left $A$-module map. Here, we regard the $A$-action on $(H/ (x))$ as trivial $A$-action. 


We take $\delta_1$ and we have the diagram 
\[ 
 \xymatrix@1{
A \otimes A \otimes A \ar[r]^{\delta_1} \ar[d]_{id \otimes x} & A \otimes A \ar[d] \\
A \otimes A \otimes A \otimes H \ar[r] & C_{\delta_1}.
} \] 

Then, we have a map $\bar{\delta_1}: Coker(\rho_1) \to C_{\delta_0}$ from the following diagram

\[
\xymatrix @R=9pt@C=6pt{
A\otimes A \otimes A \otimes A \ar[rr]^{\delta_2} \ar[dd]_-{id \otimes x} \ar[rd] & & A\otimes A \otimes A  \ar[rd]^{0} \ar[dd]^(.35){id \otimes x} | \hole  \ar[rr]^{\delta_1} & & A\otimes A \ar[rd]^{\delta_0} \ar[dd]_(.35){id \otimes x} | \hole & \\
 & 0 \ar[rr] \ar[dd] & & 0 \ar[dd]   \ar[rr] & & A \ar[dd] \\
A \otimes A \otimes A \otimes A \otimes H \ar[rr]^(.45){\delta_2 \otimes id_H} |\hole  \ar[rd] & & A \otimes A \otimes A \otimes H  \ar[rd]  \ar[rr]^(.45){\delta_1 \otimes id_H} |\hole  & & A \otimes A \otimes H \ar[rd] & \\
 & Coker(id \otimes x) \ar[rr]_-{\overline{\delta_2 \otimes id_H}}  && Coker(id \otimes x ) \ar[rr]_-{\bar{\delta_1}}  && C_{\delta_0}
 }
\]
and we know that the composition $\bar{\delta_1} \circ \overline{\delta_2 \otimes id_H}=0$. 


Then, we obtain the cylinder $C_{\bar{\delta_1}}$ of the map $\bar{\delta_1}$. 

Since $id \otimes x$ is an $A$-split map, the cylinder $C_{\bar{\delta_1}}$ has an $A$-split filtration $0 \subset A \subset Cone(\delta_0) \subset C_{\bar{\delta_1}}$ whose subquotients are isomorphic, as left $A$-modules, to $A$, $A^{\otimes 2} \otimes (H/ (x))$, and $A^{\otimes 3} \otimes (H/ (x))^{\otimes 2}$, respectively. 

Thus, we proceed to construct bar resolution inductively. 

Let $C_0=C_{\delta_0}$. Inductively, we obtain the following 

(i) $C_n \cong \text{mapping cylinder of} (\bar{\delta_n} : A^{\otimes (n+2)} \otimes (H/ (x))^{\otimes n} \to C_{n-1})$. Here we use the expedient notation $A^{\otimes (n+2)} \otimes (H/ (x))^{\otimes n}$ for the object obtained by iterated taking cokernel procedure. 

(ii) $\delta_n(\overline{\delta_{n+1} \otimes id})=0$

This assumption implies that $C_{n-1}$ is a submodule of $C_n$ and $C_n$ has an $A$-split filtration
\[
0=F^{-1} \subset F^0 \subset \cdots \subset F^{p-1} \subset F^p \subset \cdots \subset F^{n+1}=C_n , 
\]
whose subquotients $F^n / F^{n-1}$ are isomorphic to $A^{\otimes (n+2)} \otimes (H/ (x))^{\otimes n}$.

Note that $\Hom_A(P, K)$ is contructible for a projective $A$-module $P$ and an injective $H$-comodule $K$ by decomposing $P$ into finite dimensional ones and that $F^n/F^{n-1}$ is $A$-free.  Apply $\Hom_A(-, K)$ to the $A$-split short exact sequence $0 \to F^{n-1} \to F^n \to F^n / F^{n-1} \to 0$. Induction for $F_{n-1}$ and $F_n / F_{n-1}$, we can see that $\Hom_A(F_n, K)$ is contructible by \cite[Corollary 3.18]{Fari}. Thus, $\Hom_A(C_n, K)$ is contructible for an injective $H$-comodule $K$. 

Recall that the definition of cofibrant objects as in Definition~\ref{QiCof}. For a surjective stable equivalence $L \to M$, the kernel is an injective $H$-comodule $K$. We have an exact sequence $0 \to \Hom_A(C_n, K) \to \Hom_A(C_n , L) \to \Hom_A(C_n, M) \to 0$. 
Since $\Hom_A(C_n , K)$ is contructible, we have $\Hom_A(C_n, L) \cong \Hom_A(C_n, K) \oplus \Hom_A(C_n, M)$. By taking $H^\ast$-invariant part of both hand sides, we have a surjection $\Hom_A^H(C_n, L) \to \Hom_A^H(C_n, M)$. Thus, $C_n$ is cofibrant. 
\end{proof}

%
%
%

Now we define $aA=\cup_{0 \le n \le \infty} C_n$, which fits into a short exact sequence, 
$0 \to A \to aA \to \bar{p}A \to 0$ since each $C_n$ has the factor $A$ in the $A$-split filtration. 
Now we are going to show that $aA$ in the above short exact sequence is contractible as a right $H$-comodule. 

We will show this for $aA$, and the general case follows by the same argument.
The homotopy $s : A^{\otimes n} \to A^{\otimes (n+1)}$ as an extra degeneracy is a right $H$-comodule map since $A$ is an $H$-comodule algebra. 
Thus, a mapping cylinder becomes a retract of cone of the form $A^{\otimes (n+1)}\otimes H$, so that each $C_n$ is an injective right $H$-comodule. 

It is then easily seen that the map $A \otimes A \otimes k \cong A \otimes A \to A$ extends to a surjection $pA \to A$. 

Finally, we obtain the following proposition. 
\begin{prop}\label{main}
The category $\LMod_A(\mathcal{M}^H)$ inherits a model structure with respect to the three classes of maps defined as in Definition~\ref{modelstr}. 
\end{prop}
\begin{proof}
By the lemmas and the fact that all limits and colimits exists since it is a module category. 
\end{proof}

\begin{cor}
The forgetful functor $U : \LMod_A(\mathcal{M}^H) \to \mathcal{M}^H$ is the right Quillen functor. Here, we regard $\mathcal{M}^H$ as the stable model category.   
\end{cor}
\qed

\begin{rem}
Although the model structure in Proposition~\ref{main} and right-induced model structure in the paper of Hess and Shipley~\cite{HeS1} are introduced in different situations, but the model structure on the category of left $A$-module objects in $\mathcal{M}^H$ in this paper has the similar structure of right induced model structure~\cite[Definition 4.1]{HeS1}. 
\end{rem}

\subsection{Relation to the category of chain complexes of comodules}
In \cite[Section 2]{Hov}, Hovey constructed a Quillen adjunction $F : \mathcal{M}_{H^\ast} \to Ch(H)$ for a finite dimensional commutative Hopf algebra $H$, where $\mathcal{M}_{H^\ast}$ is endowed with the model structure on Frobenius category while the weak equivalences in $Ch(H)$ are "homotopy isomorphisms" as in \cite{Hov}. Let $Tk$ be a Tate resolution of the base field $k$. Recall that a Tate resolution $Tk$ of $k$ is a complex of projectives that are also injectives with no homology, such that $Z_0Tk=k$. We take a projective resolution $P_{\ast} \to k$ and an injective resolution $k \to I_{\ast}$, and let $(Tk)_n=P_{n-1}$ if $n > 0$ and $(Tk)_n=I_{-n}$ if $n \le 0$. In particular, the cycles in degree $0$ are just $k$. Then, we define $FM=Tk \otimes M$. The right adjoint $U : Ch(H) \to \mathcal{M}_{H^\ast}$ of $F$ is then defined by $UX=Z_0\Hom(Tk, X)$. Since the tensor product is over a field, $F$ preserves injections, and hence cofibrations. 
Note that the complex $Tk \otimes H$ is chain homotopy equivalent to $0$. 

By composition with the forgetful functor $U$, we have the composition of functors 
\[ 
 \xymatrix@1{
\LMod_A(\mathcal{M}^H)  \ar[r]^-{U} &   \mathcal{M}^H  \ar[r]^-{i}  & \mathcal{M}_{H^\ast}  \ar[r]^-F & Ch(H) & 
} \] 
Here, $\LMod_A(\mathcal{M}^H)$ inherits the model structure as in Proposition~\ref{main}. Note that $\LMod_A(\mathcal{M}^H)$ may not a Frobenius category in general. 
We denote by $\mathcal{M}_{H^\ast}$ the category of $H^\ast$-modules, the forgetful functor $U$ is the right adjoint and the inclusion $i$ is the Quillen functor and the fonctor $F=(-) \otimes Tk$ given by tensoring a Tate resolution of the base field is the left Quillen adjoint since $F$ preserves cofibrations and also colimits.

\section{Hopf cyclic cohomology revisited}
We will need several related combinatorial categories $\Delta$, $\Lambda_\infty$, $\Lambda_\mathbb{N}$ and $\Lambda$. 

We recall the definition of the para-cyclic category $\Lambda_\infty$. It is the full subcategory $\Lambda_\infty \subset \mathbb{Z}Poset$ consisting of all objects isomorphic to $\frac{1}{n}\mathbb{Z}$ for $n>1$. Here, $Poset$ is the category or partially ordered sets and non-decreasing maps, $\mathbb{Z}Poset$ is the category of objects in $Poset$ equipped with a $\mathbb{Z}$-action, and we consider $\frac{1}{n}\mathbb{Z}$ as an object with its natural ordering and the $\mathbb{Z}$-action given by addition. 

There is an action of $\mathbb{Z}$ on the morphism spaces of $\Lambda_\infty$, with the generator $\sigma \in \mathbb{Z}$ sending a morphism $f : \frac{1}{n}\mathbb{Z} \to \frac{1}{m}\mathbb{Z}$ to $\sigma (f) : \frac{1}{n}\mathbb{Z} \to \frac{1}{m}\mathbb{Z}$ given by $\sigma (f)=f+1$. 
By restricting $B\mathbb{Z}$ to $B \mathbb{N}$, we obtain pseudo-para-cyclic category $\Lambda_{\mathbb{N}}$. 
For any integer $p \ge 1$, we define $\Lambda_p=\Lambda_\infty / B(p\mathbb{Z})$. 
By definition, there is a $BC_p=B\mathbb{Z} / B(p\mathbb{Z})$-action on the category $\Lambda_p$, and we refer to $\Lambda=\Lambda_p / BC_p$ as the cyclic category. 
The categories $\Lambda_\infty$ and $\Lambda$ are self-dual. 

\begin{rem}
Elmendorf investigated a larger category $L$ in order to understand the self-duality of $\Lambda$. The same category as $L$ is denoted by $\Lambda_{\infty}$ by Getzler and Jones and is called the para-cyclic category. 
The point is that an object $[n]$ in $\Lambda$ can be regarded as $\mathbb{Z}/(n+1)\mathbb{Z}$ as sets. 
The isomorphism $L \to \Lambda_{\infty}$ is given by $\mathbb{Z} \times \{n\} \mapsto \frac{1}{n+1}\mathbb{Z}$. 
Nikolaus and Scholze recalled $L$ as a full subcategory of the category of posets
with $\mathbb{Z}$-actions.
The subcategory generated by $d^i_n$ and $s^j_n$ for $n \ge 0$ in $L$ is isomorphic to $\Delta$. 
 Any morphism in $L$ can be factored as compositions of $d_n^i$, $s_n^i$, $t_n$ and $t_n^{-1}$. 
This says that $L$ is obtained from $\Delta$ by adding operators $t_n$. 
\end{rem}
Especially, $\Lambda_{\infty}$ contains $\Delta$.
The category $\Lambda_{\infty}$ is obtained by adding isomorphisms $t_n : [n] \to [n]$ to $\Delta$ for $n \in \mathbb{Z}_{\ge 0}$ with relations
\begin{align*}
   t_n \circ d^i_{n-1}=d^{i-1}_{n-1} \circ t_{n-1} & (i>0)\\
   t_n \circ d^0_{n-1} =d_{n-1}^n  &   \\
  t_n \circ s^i_{n+1}=  s^{i-1}_{n+1} \circ t_{n+1} & (i>0) \\
  t_n \circ s^0_{n+1}=s^{n+1}_{n+1} \circ t^2_{n+1} &
 \end{align*}

Roughly speaking, it is obtained by dropping the condition $ t_n^{n+1}=1$ from the cyclic category. 

The subcategory of $\Lambda_{\infty}$ generated by $\Delta$ and $t_n^j$ for $j \ge 0$ is denoted by $\Lambda_{\mathbb{N}}$.    
\begin{defn}
A contravariant functor $X: \Lambda^{op}_{\infty} \to \mathcal{C}$ is called a para-cyclic object in $\mathcal{C}$. 
Covariant functors $\Lambda_{\infty} \to \mathcal{C}$ are called para-cocyclic objects. 

A para-cyclic object in $\mathcal{C}$ is a simplicial object $X$ in $\mathcal{C}$ together with morphisms $t_n : X_n \to X_n$ for $n \ge 0$ satisfying the following identities:
\begin{align*}
   d^i_{n} \circ t_n=t_{n-1} \circ d^{n-1}_{i-1} \\
   d_0^n \circ t_{n} =d_{n}^n    \\
  s_i^n \circ t_n = t_{n+1} \circ s_{i-1}^n & \\
  s_0^n \circ t_n= (t_{n+1})^2 \circ s_n^n &
 \end{align*}
Dually a para-cocycic object in $\mathcal{C}$ is a cosimplicial object $Y$ together with morphisms $t_n : Y^n \to Y^n$ for $n \ge 0$ satisfying the following identities:
\begin{align*}
   t_n \circ d^i=d^{i-1} \circ t_{n-1} & \\
   t_n \circ d^0 =d^n  &   \\
  t_n \circ s^i=  s^{i-1} \circ t_{n+1} & \\
  t_n \circ s^0=s^{n} \circ t^2_{n+1} &
 \end{align*}

Let $\Lambda_+$ be the subcategory of $\Lambda_{\mathbb{N}}$ generated by $d^n_i$ and $s^n_i$ with only $0 \le i \le n$. In other words, $\Lambda_+$ is the subcategory of the category $\Lambda$ leaving out the cyclic morphisms and the last face maps $d^n_{n+1}$ at each degree $n \ge 0$.

A pseudo-para-cyclic object in a category $\mathcal{C}$ is a functor $\Lambda^{op}_{+}\to \mathcal{C}$. 
Dually a pseudo-para-cocyclic object in a category $\mathcal{C}$ is a covariant functor $\Lambda_{+}\to \mathcal{C}$. 
Roughly speaking, this object is obtained by dropping the condition that  $d_n$ (or $d_0$) and $t_n$ are $A$-module maps. 

A cyclic object in a category $\mathcal{C}$ is defined to be a functor $X : \Lambda^{op} \to \mathcal{C}$. Morphisms between cyclic objects are natural transformations. 
Dually a functor $X : \Lambda \to \mathcal{C}$ is called a cocyclic object in $\mathcal{C}$. 

The category $\Lambda$ is generated by $d^i : [n] \to [n+1]$, $s^i : [n] \to [n-1]$ and $t_n : [n] \to [n]$ subject to certain relations. 
Especially, $\Lambda$ includes $\Delta$ as a subcategory. A cyclic object in a category $\mathcal{C}$ is a simplicial object $M$ equipped with additional morphisms $t_n : M_n \to M_n$ subject to the relations 
\begin{align*}
  \begin{cases}
   d_i \circ t_n=t_{n-1} \circ d_{i-1} & 1 \le i \le n \\
  d_0 \circ t_n =d_n &  \\
  s_i \circ t_n =  t_{n+1} \circ s_{i-1} & 1\le i \le n \\
  s_0 \circ t_n=t^2_{n+1} \circ s_n &  \\
  t_n^{n+1}=1 &
\end{cases}
 \end{align*}
\end{defn}
\begin{exa}
For a $k$-algebra $A$ over a commutative ring $k$, the cyclic bar construction $N_{\bullet}^{cyc}$ is a cyclic $k$-module. 
\end{exa}
\begin{lemma}\label{dn}
Let $X$ be a simplicial object together with invertible morphisms $t_n : X_n \to X_n$ for all $n \ge 0$ satisfying the following conditions: 
\begin{align*}
   d_i=t^i_{n-1} \circ d_0 \circ t^{-i}_{n} & \\
   s_i=  t^i_{n+1} \circ s_0 \circ t^{-i}_{n} & 
 \end{align*}
 Then $X$ is a para-cyclic object.
\end{lemma}
\begin{proof}
For $i>0$, we have the equations 
\begin{align*}
   d_i \circ t_n & =t^i_{n-1} \circ d_0 \circ t_n^{-i} \circ t_n \\
 &  = t_{n-1} \circ (t_{n-1}^{i-1}) \circ d_0 \circ t_n^{-(i-1)}  \\
&  = t_{n-1} \circ d_{i-1} \\
  s_i \circ t_n & = t_{n+1}^i \circ s_0 \circ t_n^{-i} t_n \\
  & = t_{n+1} \circ t_{n+1}^{i-1} \circ s_0 \circ t^{-(i-1)}_n  \\
& =t_{n+1} \circ s_{i-1}
 \end{align*}

When $i=0$, we have the equations
\begin{align*}
   d_0 \circ t_n & =t^{-n}_{n-1} \circ t^n_{n-1} \circ d_0 \circ t_n^{-n} \circ t^n_n \\
 & =d_n \circ t_{n}^n \\
  s_0\circ t_n & = t_{n+1}^{n+2} \circ s_0 \circ t_n \\
  & = t^2_{n+1} \circ t^n_{n+1} \circ s_0 \circ t_n^{-(n+1)} \circ tn  \\
& = t_{n+1}^2 \circ s_n  
   \end{align*}

Thus $X$ is a para-cyclic object. 
\end{proof}

\subsection{Vanishing of Hopf-cyclic theory}
We recall Doi's results in \cite{Doi2} and \cite{Doi} that determined the form of projective $A$-modules, e.g., a projective $A$-module with right $H$-comodule structure is a direct summand of $A \otimes_k V$, where $V$ is a right $H$-comodule. He also considerd a certain algebra map $H \to A$ called a total integral, and projectivity was related to the existence of total integral as follows. 

\begin{enumerate}[(i)]
\item Let $H$ be a Hopf algebra over a field $k$ and $A$ a right $H$-comodule algebra. If a map $\phi : H \to A$ of right $H$-comodules with $\phi (1)=1$, i.e., if there exists total integral, then any right relative $(A, H)$-Hopf module is an injective $H$-comodule.  
\item Furthermore, if the above right $H$-comodule map $\phi : H \to A$ is an algebra map, then the fundamental theorem for right relative $(A, H)$-Hopf modules is true.  
\item Let $H$ be a Hopf algebra over a commutative ring $R$ and $A$ a right $H$-comodule algebra. Doi~\cite{Doi} gave a sufficient condition such that an epimorphism of $(A, H)$-Hopf module splits if it splits $A$-linearly. As an application in the case when $R$ is a field, he got that an $(A, H)$-Hopf module is finitely generated projective as an $A$-module if and only if it is a Hopf module direct summand of $A \otimes M$ for some finite dimensional $H$-comodule $M$.
\end{enumerate}

In our case, Hopf cyclic theory will be considered in the base category $\mathcal{M}^H$. If we consider an $H$-module in $\mathcal{M}^H$, it is automatically compatible with the right $H$-comodule structure, so that it is contructible by Doi's result. By fundamental theorem of Hopf modules, any Hopf-cyclic objects in $\mathcal{M}^H$ whose coefficients in a Hopf module is zero in $\underline{\mathcal{M}^H}$. 
\begin{cor}
A Hopf (co)cyclic homology with coefficient in a $H$-module / comodule vanishes in the stable homotopy category $\underline{\mathcal{M}^H}$. 
\end{cor}

\subsection{Hopf cyclic for $H$-comodule bialgebra $A$}
\begin{defn}
For a bialgebra $B$, a stable $B$-module / comodule is a $B$-module and $B$-comodule with respect to the structure maps $m_M : B \otimes_k M \to M$ and $\rho_M : M \to B \otimes_k M$, respectively, such that $m_M \circ \rho_M : M \to B \otimes_k M \to M$ is the identity. 
\end{defn}

We let $A$ a bialgebra in $\mathcal{M}^H$ with respect to $\otimes_k$. 

\begin{defn}
 Let $A$ be a right $H$-comodule bialgebra and $M$ a stable left
 $A$-module/comodule. 
 Define $T_{n}(A,M) = A^{\otimes (n+1)} \otimes M$ for $n\ge 0$.
 An element $a_{0}\otimes a_{1}\otimes\cdots \otimes a_{n} \otimes m$ of 
 $T_{n}(A,M)$ is denoted by $[a_{0}|a_{1}|\cdots|a_{n}]m$. 
 
 Define face, degeneracy, and cyclic operators by
 \begin{align*}
  d_{i}([a_{0}|a_{1}|\cdots|a_{n}]m) & =
  \begin{cases}
   [a_{0}|a_{1}| \cdots | a_{i}a_{i+1}| \cdots |  a_{n}]m , & i\neq n \\ 
     [m^{(-1)}a_{n}a_0 | a_{1} |\cdots |a_{n-1}]m^{(0)}, & i=n 
  \end{cases} \\
  s_{i}([a_{0}|a_{1}|\cdots | a_{n}]m) & =
  [a_{0}|a_{1}| \cdots |a_{i}| 1 | a_{i+1} | \cdots | a_{n}]m \\
  t_{n}([a_{0}|a_{1}|\cdots | a_{n}]m ) & =
   [m^{(-1)}a_{n}|a_{0}| \cdots| a_{n-1}]m^{(0)} . 
 \end{align*}
\end{defn}
\begin{defn}
 Define a left $H$-comodule structure on $T_n(A,M)$ 
 by the diagonal coaction 
\end{defn}
\begin{prop}\label{cyclic}
The above maps define a structure of pseudo-para-cyclic right $H$-comodule on
 $T_\bullet(A,M)$. 
\end{prop}
\begin{proof}
Note that the structure maps of bialgebra $A$ and $A$-module/comodule $M$ are all compatible with right $H$-comodule structure, so that $T_\bullet(A, M)$ is a right $H$-comodule.  
 Let us verify some nontrivial part of the simplicial identities. For $i<j$
 \begin{align*}
  (d_{i}\circ d_{j})([a_{0}|a_{1}|\cdots | a_{n}]m ) 
  & = 
  \begin{cases}
   d_{i}([a_{0}|a_{1}| \cdots | a_{j}a_{j+1} |
   \cdots | a_{n}]m), & j< n \\
   d_{i}([m^{(-1)}a_{n}a_0 |a_{1}|\cdots | a_{n-1}]m^{(0)}, & j=n 
  \end{cases} \\
  & = 
  \begin{cases}
   [a_{0}|a_{1}| \cdots | a_{i}a_{i+1} | \cdots
   | a_{j}a_{j+1} | \cdots | a_{n}]m  , & i+1<j<n \\
   [a_{0}|a_{1}| \cdots | a_{i}a_{i+1}a_{i+2} |\cdots | a_{n}]m, &
   i+1=j<n \\ 
   [m^{(-1)}a_{n}a_{0}a_{1} |a_{2}| \cdots | a_{n-1}]m^{(0)}, &
   i=0, j=n \\ 
    [m^{(-1)}a_{n}a_{0} |a_{1} | \cdots | a_{i}a_{i+1} | \cdots
   | a_{n-1}]m^{(0)}, & 1<i+1<j=n \\
    [m^{(-1)}a_{n-1} m^{(-2)}a_{n} a_0 | a_{1}|
   \cdots | a_{n-2}]m^{(0)}, & 
   i+1=j=n  
  \end{cases} \\
  (d_{j-1}\circ d_{i})([a_{0}|a_{1}|\cdots | a_{n}] m) 
  & = 
  d_{j-1}([a_{0}|a_{1}| \cdots | a_{i}a_{i+1} |
  \cdots | a_{n}]m) \\
  & = 
  \begin{cases}
   [a_{0}|a_{1}| \cdots | a_{i}a_{i+1}a_{i+2} |
   \cdots | a_{n}]m & i+1=j<n \\ 
   [a_{0}|a_{1}| \cdots | a_{i}a_{i+1}| \cdots|
    a_{j}a_{j+1} | \cdots | a_{n}]m & i+1<j<n \\ 
   [m^{(-1)}a_{n}a_{0}a_{1} |a_{2}| \cdots | a_{n-1}]m^{(0)} & i=0,j=n \\
   [m^{(-1)}a_{n}a_{0} |a_{1} | \cdots | a_{i}a_{i+1} | \cdots
   | a_{n-1}]m^{(0)}, &
   1<i+1<j=n \\ 
    [m^{(-1)}a_{n-1} m^{(-2)}a_{n} a_0 | a_{1}| \cdots|
   a_{n-2}]m^{(0)}, & i=j=n.  
  \end{cases}
 \end{align*}
 Let us check the relations between $t_{n}$ and $d_{i}$. 
 \begin{align*}
  (d_{i}\circ t_{n})([a_{0}|a_{1}|\cdots | a_{n}]m) 
  & = d_{i}([m^{(-1)}a_{n}|a_{0}| \cdots | a_{n-1}]m^{(0)}) \\
  & = 
  \begin{cases}
  [m^{(-1)}a_{n}|a_{0}| \cdots | a_{i}a_{i+1} |
   \cdots | a_{n-1}]m^{(0)} , & i<n \\
  [m^{(-1)}a_{n-1}m^{(-2)} a_{n}|a_{0}| \cdots
   | a_{n-2}]m^{(0)}, & i=n
  \end{cases} \\
  (t_{n-1}\circ d_{i-1})( [a_{0}|a_{1}|\cdots | a_{n}]m) 
  & = 
  \begin{cases}
  [m^{(-1)}a_{n}|a_{0}| \cdots | a_{i}a_{i+1} |
   \cdots | a_{n-1}]m^{(0)} , & i<n \\
  [ m^{(-1)}m^{(-2)}a_{n-1} a_{n}|a_{0}| \cdots
   | a_{n-2}]m^{(0)}, & i=n
  \end{cases}
 \end{align*}
 Here, $m^{(-1)}m^{(-2)}a_{n-1} a_{n}$ may not be equal to $m^{(-1)}a_{n-1}m^{(-2)} a_{n}$ but at least we checked that $T_\bullet (A, M)$ is pseudo-para-cyclic. 

For a left $H$-comodule structure on $T_\bullet(A, M)$, it can be verified since $d_i$, $s_i$ and $t_n$ are $H$-colinear maps. 
\end{proof}

The reasons why the above pseudo-para-cyclic $T_\bullet(A, M)$ is not para-cyclic are that $d_n$ does not commute with $t_n$ and that the map $t_n$ is not invertible.

Note that an $H$-module algebra $A$ means an $H$-module and $k$-algebra with $H$-linear unit and multiplication. This is equivalent to original definition of $H$-module algebra that is an $H$-module and $k$-algebra with compatibility~\cite[Section 4.1]{Mont}. The same is true for $H$-comodule algebra, and also $H$-(co)module bialgebra. 
However, for an $H$-(co)module $k$-Hopf algebra, the antipode would be a $k$-linear map. 

\begin{lemma}
If $A$ is a right $H$-comodule bialgebra which is also $k$-Hopf algebra with bijective $H$-colinear antipode $S$, then the map $t_n$ is an $H$-colinear isomorphism. The inverse map is given by $ t^{-1}_{n}([a_{1}|a_{2}|\cdots | a_{n}]m)=([a_{1}|a_{2}|\cdots | a_{n} | S(m^{(-1)})a_0]m^{(0)})$. 

Moreover, we have $t_n^{n+1}=id$ as $H$-comodule maps. 
\end{lemma}
\begin{proof}
The first part follows from Lemma~\ref{dn}. 

Since $M$ is stable and by using Lemma~\ref{314}, we obtain the equation $t^{n+1}_{n}([a_{0}|a_{1}|\cdots | a_{n}]m)=([(a_{0}|a_{1}|\cdots | a_{n})]\epsilon(m^{(-1)})m^{(0)})=[a_{0}|a_{1}|\cdots | a_{n}]m$. 
Therefore, we have $t_n^{n+1}=id$.  
\end{proof}
\begin{defn}[cf.\cite{Kay}, p347]
 Let $H$, $A$, and $M$ be as above. Define $Q^A_{\bullet}(A,M)$ by the largest pseudo-para-cyclic subcomodule of $T_n(A, M)$ which is a cyclic right $H$-comodule. 
\end{defn}
In \cite{Kay}, $Q^A_{\bullet}(A,M)$ is called the co-approximation of $T_n(A, M)$. 
Note that $Q^A_n(A, M)$ is also a para-cyclic $H$-comodule, and actually a cyclic $H$-comodule.  

We use the universality of Kan extension. 
We recall that $Q^A_n(A, M)$ is the right Kan extension. 

Assume that $A$ is a right $H$-comodule bialgebra which is also $k$-Hopf algebra with bijective $H$-colinear antipode $S$ and that $M$ is stable. We obtain a cyclic right $H$-comodule $T_\bullet(A, M)$ in $\mathcal{M}^H$. On the other hand, we also have a cyclic right $H$-comodule $Q^A_\bullet(A,M)$ which is the right Kan extension along $\Lambda_+ \to \Lambda$. 
\begin{prop}
There is always a map $T_\bullet(A, M) \to Q^A_\bullet(A, M)$ of right $H$-comodules in $\mathcal{M}^H$. 
\end{prop}
\begin{proof}
It follows from the universality of right Kan extension. 
\end{proof}
Note that the $H$-coinvariant part of $T_\bullet (A, M)$ is an analogy of Connes and Moscovich's Hopf-cyclic homology. 
If we endow $M$ the trivial $A$-action and the trivial $A$-coaction, from the equations in the proof of Proposition~\ref{cyclic}, we have the non-Hopf cyclic structure. Therefore, we can compare Hopf-cyclic modules with non-Hopf cyclic modules as an analogue of the characteristic map in the paper of Connes and Moscovici as follows. 
\begin{cor}\label{main2}
There is a map $Q^A_\bullet(A, k) \to Q^A_\bullet(A, A)$ of right $H$-comodules in $\mathcal{M}^H$ induced from the unit map. 
\end{cor}
\qed

\begin{rem}\label{shift}
We can define a cyclic $k$-module $Q_\bullet(A,M)$ by taking $H$-coinvariant submodule $Q_\bullet^A(A, M)^{coH}$ of $Q_\bullet^A(A, M)$. On the other hand, only object $A^{\otimes n} \otimes M$ in each degree in $Q^A_\bullet (A, M) $ admits the diagonal $A$-action. As an expansion of an idea of taking $k$ in Corollary~\ref{main2}, we will take an $A$-coinvariant part of $M$. 
Assume that $M$ is a Hopf $A$-module in $\mathcal{M}^H$. If we take $A$-coinvariant part of $M$, then we have $A \otimes M^{coA} \cong M$ by the fundamental theorem. The isomorphism is $H$-colinear left $A$-module map. Therefore, we have $A^{\otimes n} \otimes M^{coA} \cong A^{\otimes (n-1)} \otimes M$ in each degree, and the degree shift of $Q_\bullet^A(A, M)$ appeared. 
\end{rem}

If $H$ is a finite dimensional commutative Hopf algebra, we have the composition of functors 
\[ 
 \xymatrix@1{
\LMod_A(\mathcal{M}^H)  \ar[r]^-{U} &   \mathcal{M}^H  \ar[r]^-{i}  & \mathcal{M}_{H^\ast}  \ar[r]^-F & Ch(H) , 
} \] 
where $\mathcal{M}_{H^\ast}$ is the category of $H^\ast$-modules, the forgetful functor $U$ is the right adjoint and the inclusion $i$ is the Quillen functor and the fonctor $F=(-) \otimes Tk$ given by tensoring a Tate resolution of the base field is the left adjoint. 

We send the (co)cyclic $H$-comodule under the functor $F$. By \cite{Kay}, the approximation theorem of cocyclic comodule theory can be regarded as a (local) left Kan extension and the coapproximation as right Kan extension. 
Therefore, Hopf-(co)cyclic homology in $\mathcal{M}^H$ can be calculated via the associated Hopf-(co)cyclic theory in $Ch(H)$.

\bibliographystyle{amsplain} \ifx\undefined\bysame
\newcommand{\bysame}{\leavemode\hbox to3em{\hrulefill}\,} \fi

\end{document}